\documentclass[a4paper,12pt]{amsart}

\usepackage{hyperref}

\def\today{\ifcase \month \or
   January \or February \or March \or April \or
   May \or June \or July \or August \or
   September \or October \or November \or December \fi
   \space\number\day , \number\year}

\newcommand{\R}{\mathbb{R}}
\newcommand{\C}{\mathbb{C}}

\newtheorem{theorem}{Theorem}[section]

\newtheorem{remark}{Remark}[section]

\begin{document}



\title[Stability of $N$-times integrated semigroups]{On stability of $N$-times integrated semigroups with nonquasianalytic growth}

\author[J. E. Gal\'e, M. M. Mart\'inez and P. J. Miana]{Jos\'e E. Gal\'e, Mar\'ia M. Mart\'inez and Pedro J. Miana}
\address{Departamento de Matem\'{a}ticas e IUMA,
Universidad de Zaragoza, \linebreak 50009, Zaragoza, Spain.}

\email{gale@unizar.es (J.E. Gal\'{e}), martinezmartinezm@gmail.com (M.M. Mart\'{i}nez), pjmiana@unizar.es (P.J. Miana).}

\thanks{The research has been partially supported  by Project MTM2013-42105-P, DGI-FEDER, of the MEYC; 
Project E-64, fondos FEDER, D.G. Arag\'on.}

\begin{abstract}
We discuss the behaviour at infinity of $n$-times integrated semigroups with nonquasianalytic growth and invertible generator. The results obtained extend in this setting a theorem of O. El Mennaoui on stability of bounded once integrated semigroups, and (partially) a theorem of Q. P. V$\tilde{\rm u}$ on stability of $C_0$-semigroups.
\end{abstract}

\maketitle
\section {Introduction}

Let $A$ be a closed operator on a Banach space $X$ with domain $D(A)$.
The solution $u\colon[0,\infty)\to X$ of the Cauchy equation
\begin{equation}\label{Cauchy}
u'(t)=Au(t),\,\, t\ge0; \quad u(0)=x\in D(A)
\end{equation}
is given by $u(t)=T_0(t)x$ where $T_0(t):=e^{tA}$, $t\ge 0$, is the $C_0$-semigroup generated by $A$, provided $A$ satisfies the Hille-Yosida condition; see \cite[Section 3.1]{[ABHN]}. There still are other important cases where $A$ does not satisfy that condition but it is the generator of an exponentially bounded {\it $n$-times integrated semigroup} in the following sense:

There exist a family $(T_n(t))_{t\ge0}$ of bounded operators on $X$ and $C,w\ge0$ such that
$\Vert T_n(t)\Vert\leq Ce^{wt}$, $t\ge0$, and
$$
(\lambda-A)^{-1}x=\lambda^n\int_0^\infty e^{-t\lambda} T_n(t) x\, dt, \quad \Re \lambda>w, \quad x\in X.
$$
Every $C_0$-semigroup is a $0$-times integrated semigroup; for more information on integrated semigroups and examples see \cite[Sections 3.2 and 8.3]{[ABHN]}, \cite{[ANS]} and references therein. 

Put $\displaystyle u(t):=(d/ dt)^nT_n(t)x$, so that  
$$
T_n(t)x=\int_0^t {(t-s)^{n-1}\over (n-1)!}\ u(s)x\ ds,\qquad x\in X.
$$
Then the function $u(t)$ 
is the unique solution to the equation (\ref{Cauchy}) and the limit $\lim_{t\to\infty}T_n(t)$ -or alternatively its ergodic version 
$\lim_{t\to\infty}t^{-n}T_n(t)$- 
reflects the asymptotic behaviour of the solution $u$ at infinity. In this respect,
when $n=0$, a uniformly bounded $C_0$-semigroup $(T_0(t))_{t\ge0}$ is said to be 
stable if
$$
\lim_{t\to\infty}T_0(t)x=0, \quad x\in X.
$$

The stability of uniformly bounded $C_0$-semigroups on Banach spaces, under certain spectral assumptions on their infinitesimal generators, is proven in \cite{[AB]} and \cite{[LV]}, with different proofs and independently one from each other. We refer to this stability result as the Arendt-Batty-Lyubich-V$\tilde{\rm u}$ theorem. It states that 
$(T_0(t))_{t\ge0}$ is stable whenever
$$
\sigma(A)\cap i\R\,\, {\hbox {is countable and}}\,\, \sigma_P(A^*)\cap i\R=\emptyset.
$$
Here $\sigma(A)$ is the spectrum of the generator $A$ and $\sigma_P(A^*)$ is the point spectrum of the adjoint operator $A^*$ of $A$. For the asymptotic behaviour and stability of operator semigroups we refer the reader to \cite{[CT]} and \cite{[vN]}.

It seems interesting to have a result like the Arendt-Batty-Lyubich-V$\tilde{\rm u}$ theorem for $n$-times integrated semigroups. In this setting, a notion of stability has been defined for once integrated semigroups as follows. Suppose that 
$(T_0(t))_{t\ge 0}$ is a $C_0$-semigroup and let $(T_1(t))_{t\ge 0}$ denote the (trivial, say) once integrated semigroup induced by $(T_0(t))_{t\ge 0}$, which is defined by
$$
T_1(t)x:=\int_0^t T_0(s)x\ ds, \quad x\in X.
$$
By a well known property of $C_0$-semigroups,
$$
T_0(t)x-x=AT_1(t)x(=T_1(t)Ax), \quad \hbox { for } x \in D(A).
$$
Assume in addition that the $C_0$-semigroup $(T_0(t))_{t\ge 0}$ is stable. Then, {\it provided $A$ is invertible}, one gets that there exists the limit
$$
\lim_{t\to\infty}T_1(t)x=-A^{-1}x, \quad x\in\overline{D(A)}.
$$
Motivated by this observation, a (nontrivial) general once integrated semigroup $T_1(t)$ is called {\it stable} in \cite[p. 363]{[EM]} when $\lim_{t\to\infty} T_1(t)x$ exists for every $x\in\overline{D(A)}$. Moreover, it is also shown in \cite[Prop. 5.1]{[EM]} that  if a once integrated semigroup $T_1(t)$ is stable in the sense of that definition then $A$ must be invertible, which is to say that $0$ belongs to the resolvent set $\rho(A)$ of $A$. The following result is \cite[Theorem 5.6]{[EM]}. It gives a version of the
Arendt-Batty-Lyubich-V$\tilde{\rm u}$ theorem for once integrated semigroups.

\begin{theorem}\label{integrastable}
Let $A$ be the generator of a once integrated semigroup $(T_1(t))_{t\ge 0}$ such that 
$\sup_{t>0}\Vert T_1(t)\Vert< +\infty$. Assume in addition that
$\sigma(A)\cap i \R$ is countable, $\sigma_P(A^*)\cap i \R=\emptyset$ and $ 0\in\rho(A)$. 
Then $(T_1(t))_{t\ge 0}$ is stable.
\end{theorem}

The boundedness condition assumed on $(T_1(t))_{t\ge 0}$ in Theorem \ref{integrastable} looks somehow restrictive: For a uniformly bounded $C_0$-semigroup $(T_0(t))_{t\ge 0}$ and its $n$-times integrated semigroup
$$
T_n(t)x={1\over(n-1)!}\int_0^t (t-s)^{n-1}T_0(s)x\ ds, \qquad t>0, \, x\in X,
$$
the derived boundedness condition on $(T_n(t))_{t\ge 0}$ which is to be expected from the integral expression is
$\sup_{t>0} t^{-n}\Vert T_n(t)\Vert<+\infty$, so that for $n=1$ is
$\sup_{t>0} t^{-1}\Vert T_1(t)\Vert<+\infty$ instead of boundedness.

The purpose of this note is to extend Theorem \ref{integrastable} to $n$-times integrated semigroups for every natural $n$ and a fairly wide boundedness condition involving nonquasianalytic weights.
We say that a positive measurable locally bounded function $\omega$ with domain $\R$ or $[0,\infty)$ is a weight if $\omega(t)\ge1$ and
$\omega(s+t)\le\omega(s)\omega(t)$ for all $t,s$ in its domain. A weight $\omega$ on $[0,\infty)$ is called nonquasianalytic if
$$
\int_0^\infty{\log \omega(t)\over t^2+1}\ dt<\infty.
$$
As in \cite[Section 1]{[V2]} we assume that
$\liminf_{t\to\infty}\omega(t)^{-1}\omega(s+t)\ge1$ for all $s>0$. Then one can define the function $\widetilde\omega$ on $\R$ given by
$$
\widetilde\omega(s):=\limsup_{t\to\infty}{\omega(t+s)\over\omega(t)}, s\ge0\ ,
\hbox{ and }
\widetilde\omega(s):=1, s<0,
$$
is a weight function. Clearly, $\widetilde\omega(t)\le\omega(t)$ for every $t\ge0$.

Our main result is the following. In the statement, and throughout the paper, the symbol $\lq\lq\sim"$ in $a(t)\sim b(t)$ as $t\to\infty$ means that
$\lim_{t\to\infty} b(t)^{-1}a(t)=c>0$ as $t\to\infty$.

\begin{theorem}\label{qnstable}
Let $A$ be the generator of a n-times integrated semigroup $(T_n(t))_{t\ge 0}$ such that
$\sigma(A)\cap i \R$ is countable, $\sigma_P(A^*)\cap i \R=\emptyset$ and $ 0\in\rho(A)$. Assume that
$$
\sup_{t\ge1}\omega(t)^{-1}\Vert T_n(t)\Vert< +\infty,
$$
for some nonquasianalytic weight $\omega$ on $[0,\infty)$ for which $\widetilde\omega(t)=O(t^k)$, as $t\to\infty$, for some $k\ge0$.

We have:
\begin{itemize}
\item[{\rm(i)}] If $\omega(t)^{-1}=o(t^{-n+1})$ as $t\to\infty$, then
$$
\lim_{t\to\infty}\omega(t)^{-1}T_n(t)x=0, \quad x\in \overline{D(A^n)}.
$$
\item[{\rm(ii)}] If $\omega(t)\sim t^{n-1}$ as $t\to\infty$, then
$$
\lim_{t\to\infty}t^{-n+1}T_n(t)x=-{1\over (n-1)!}A^{-1}x, \quad x\in \overline{D(A^n)}.
$$
\end{itemize}
\end{theorem}

\begin{remark}\label{extMennaoui}
\normalfont
For $n=1$, Theorem \ref{qnstable} (ii) is
\cite[Theorem 5.6]{[EM]}. So any $n$-times integrated semigroup $(T_n(t))_{t\ge 0}$ satisfying the equality of Theorem 
\ref{qnstable} (ii) might well be called {\it stable}. Alternatively, the ergodic type equality $\lim_{t\to\infty}\omega(t)^{-1}T_n(t)x=0$, $x\in \overline{D(A^n)}$, for $\omega(t)\sim t^{n}$ at infinity, defines a property on $(T_n(t))_{t\ge 0}$ which corresponds to stability of $C_0$-semigroups when $n=0$. Then one could say that a $n$-times integrated semigroup satisfying Theorem \ref{qnstable} (i) for
$\omega(t)\sim t^{n}$ as $t\to\infty$ is {\it stable of order $n$}, and {\it stable under $\omega$} in general.
\end{remark}

\medskip
\section {Proof of Theorem 1.2}
\label{proof}

In order to establish Theorem \ref{qnstable} one needs to extend 
\cite[Theorem 7]{[V2]} and \cite[Theorem 5.6]{[EM]}. Firstly, and more precisely, 
Theorem \ref{mejorVu} below is 
an improvement of \cite[Theorem 7]{[V2]}, which is in turn an extension of the Arendt-Batty-Lyubich-V$\rm{\tilde u}$ theorem. In fact \cite[Theorem 7]{[V2]}  is recovered in the particular case that the operator $R$ is the identity operator in Theorem \ref{mejorVu}. The case $\beta(t)\equiv 1$ in Theorem \ref{mejorVu} appears in \cite[Remark 3.3]{[AB]}.

For a Banach space $(Y, \Vert \cdot\Vert)$, let
$\mathcal B(Y)$ denote the Banach algebra of bounded operators on $Y$.

\begin{theorem}\label{mejorVu}
Let $(U(t))_{t\ge0}\subset\mathcal B(Y)$ be a  $C_0$-semigroup of positive exponential type with generator $L$. Let $\beta$ be a nonquasianalytic weight on $[0,\infty)$ such that $\widetilde\beta(t)=O(t^k)$ as $t\to\infty$, for some $k\ge0$.
Assume that there exists $R\in\mathcal B(Y)$ such that
$U(t)R=RU(t)$ for all $t\ge0$ and $\Vert U(t)R\Vert\le\beta(t)$ for $t\ge0$.

If $\sigma(L)\cap i\R$ is countable and $\sigma_P(L^*)\cap i\R=\emptyset$
then
$$
\lim_{t\to\infty}{1\over\beta(t)}U(t)Ry=0, \quad y\in Y.
$$
\end{theorem}

\begin{proof}
The overall argument goes along similar lines as in \cite[Theorem 7]{[V2]}, lemmata included. Next, we outline that argument for convenience of prospective readers and give details, when necessary, to extend the corresponding assertions to our setting. 

Put
$$
q(y):=\limsup_{t\to\infty} \beta(t)^{-1}\Vert U(t)Ry\Vert, \quad y\in Y.
$$
Then $q$ is a seminorm on $Y$ such that $q(y)\le\Vert y\Vert$ for all $y\in Y$. Moreover, $q(U(s)y)\le\widetilde\beta(s)q(y)$ for every $s\ge0$, $y\in Y$, and so
$N:=\{y\in Y:q(y)=0\}$ is a $U(t)$-invariant closed subspace of $Y$. Hence
one can define a norm $\widehat q$ on $Y/N$ given by
$$
\widehat q(\pi(y)):=q(y),\quad y\in Y,
$$
and an operator $\widehat U(t)$ on $Y/N$ given by
$$
\widehat U(t)(\pi (y)):=\pi(U(t)y),\quad y\in Y, t\ge0,
$$
where $\pi$ is the projection $Y\to Y/N$.

It is straightforward to show that $(\widehat U(t))_{t\ge0}$ is a strongly continuous semigroup in the norm $\widehat q$ on $Y/N$.
Let $(Z, \Vert\cdot\Vert_Z)$ be the $\widehat q$-completion of $Y/N$, and let $V(t)$ be the continuous extension on $Z$ of $\widehat U(t)$ for all $t>0$. Then:

\begin{itemize}
\item[{\rm(a)}]
$\Vert\pi(y)\Vert_Z
=\limsup_{t\to\infty}{1\over\beta(t)}\Vert U(t)Ry\Vert$ for $y\in Y$. This is obvious.
\item[{\rm(b)}] $\Vert V(t)\Vert_{Z\to Z}\le\widetilde\beta(t)$, $t\ge0$, and from this fact one readily obtains that $(V(t))_{t>0}$ is a  $C_0$-semigroup in $\mathcal B(Z)$. The above bound follows by continuity and density from the estimate
$$
\begin{aligned}
\widehat q(\widehat U(t)\pi (y))&
=\widehat q(\pi(U(t)y))=q(U(t)y) \\
&\le\widetilde\beta(t)q(y)\le\widetilde\beta(t)\widehat q(\pi (y)), \quad y\in Y, t\ge0.
\end{aligned}
$$
\item[{\rm(c)}] $\Vert V(t)z\Vert_Z\ge\Vert z\Vert_Z$ for all $z\in Z$: For $y\in Y$ and $t\ge0$,
$$
\widehat q(\widehat U(t)\pi (y))
=\limsup_{t\to\infty}{\beta(t+s)\over\beta(t)}{\Vert U(t+s)Ry\Vert_Y\over\beta(t+s)}
\ge \widehat q(\pi(y)).
$$
Then we apply continuity and density.
\item[{\rm(d)}] $V(t)\circ\pi=\pi\circ U(t)$ ($t\ge0$) and then one easily obtains that
$\pi(D(L))\subseteq D(H)$ and $H\circ\pi=\pi\circ L$ on $D(L)$, where $H$ is the infinitesimal generator of $(V(t))_{t\ge 0}$.
\item[{\rm(e)}] $\sigma(H)\subseteq\sigma(L)$: By hypothesis, $(U(t))_{t\ge 0}$ is of exponential type $\delta>0$ whence, as is well known, for $y\in Y$ and
$\lambda\in\C$, $\Re \lambda>\delta$,
$$
R(\lambda,L)y:=-(\lambda-L)^{-1}y
=-\int_0^\infty e^{-\lambda t} U(t)y\ dt.
$$
Similarly, since
$\Vert V(t)\Vert_{Z\to Z}\le\widetilde\beta(t)$ for all $t\ge0$, the semigroup $(V(t))_{t\ge 0}$ is of exponential type $0$, and therefore we have for $z\in Z$ and $\lambda\in\C$, $\Re \lambda>0$,
$$
R(\lambda,H)z:=-(\lambda-H)^{-1}z
=-\int_0^\infty e^{-\lambda t} V(t)z\ dt.
$$

On the other hand, $R$ commutes with $U(t)$, $t\ge 0$, by assumption and so $R$ commutes with $R(\lambda,L)$ for 
$\Re \lambda>\delta$. Then
$q(R(\lambda,L)y)\le\Vert R(\lambda,L)\Vert q(y)$ for all $y\in Y$, which implies that $N$ is $R(\lambda,L)$-invariant. 
Hence one can define the bounded operator
$\widehat R(\lambda,L)$ on $Z$ given by
$\widehat R(\lambda,L)(\pi(y)):=\pi\left(R(\lambda,L)y\right)$, $y\in Y$. 
Thus,
$$
\begin{aligned}
\widehat R(\lambda,L)\pi(y)&=\pi\left(R(\lambda,L)y\right)
=-\int_0^\infty e^{-\lambda t}\pi\left(U(t)y\right) dt \\
&=-\int_0^\infty e^{-\lambda t}V(t)\pi(y) dt
=R(\lambda,H)\pi(y)
\end{aligned}
$$
where (d) has been applied in the last but one equality. Hence
$\widehat R(\lambda,L)=R(\lambda,H)$, for $\Re\lambda>\delta$.

Now, for $\Re\lambda>\delta$ and any $\mu\in\rho(L)$, by using the resolvent identity
$$
R(\lambda,L)-R(\mu,L)=(\lambda-\mu)R(\lambda,L)R(\mu,L)
$$
on $Y$ and its corresponding identity for $\widehat R(\lambda,L)$ and $\widehat R(\mu,L)$ on $Z$, one readily finds that there exists $R(\mu,H)$ with
$$
R(\mu,H)=\widehat R(\mu,L),
$$
see more details in \cite[p. 234]{[V2]}. Thus $\mu\in\rho(H)$. Hence $\rho(L)\subseteq\rho(H)$
as we claimed.
\item[{\rm(f)}] $\sigma_P(H^*)\subseteq\sigma_P(L^*)$. This is straightforward to see, using restrictions of functionals; see \cite[p. 234]{[V2]}.
\end{itemize}

Suppose, if possible, that $Z\not=\{0\}$. By (e) above, we have that $\sigma(H)\cap i\R$ is countable and then $i\R\setminus\sigma(H)\not=\emptyset$. So, by (c) above and \cite[Lemma 2]{[V2]}, the $C_0$-semigroup $(V(t))_{t\ge 0}$ can be extended to a $C_0$-group
$(\widetilde V(t))_{t\in \R}$  such that
$\Vert\widetilde V(-t)\Vert_{Z\to Z}\le1$ ($t>0$) and $\Vert\widetilde V(t)\Vert_{Z\to Z}=O(t^k)$, as $t\to+\infty$. Also, $\sigma(H)$ is nonempty by (b) above and \cite[Lemma 5]{[V2]}.

Then reasoning as in \cite[Theorem 7]{[V2]}  one gets
$\sigma_P(H^*)\cap i\R\not=\emptyset$ whence
$\sigma_P(L^*)\cap i\R\not=\emptyset$ by (f) above. This is a contradiction and so we have proved that $Z=\{0\}$. By (a) above, we get the statement.
\end{proof}

The following theorem is the quoted extension of \cite[Theorem 5.6]{[EM]}.

\begin{theorem}\label{preResult}
Let $\omega$ be a nonquasianalytic weight such that $\widetilde\omega$ is of polynomial growth at infinity. 
Let $(X, \Vert \cdot \Vert)$ be a Banach space and $(T_n(t))_{t\ge0}$ be a $n$-times integrated semigroup in $\mathcal B(X)$ with generator $(A,D(A))$ such that
$\Vert T_n(t)\Vert\le\omega(t)$, $t\ge0$. Let assume that
$\sigma(A)\cap i\R$ is countable and $\sigma_P(A^*)\cap i\R=\emptyset$.

For every $\mu>0$ we have:
\begin{itemize}
\item[{\rm(i)}] If $\omega(t)^{-1}=o(t^{-(n-1)})$, as $t\to\infty$, then
$$
\lim_{t\to\infty}\omega(t)^{-1}T_n(t)A^n(\mu-A)^{-2n}x=0, \quad x\in X.
$$
\item[{\rm(ii)}] If $\omega(t)\sim t^{n-1}$, as $t\to\infty$, then
$$
\lim_{t\to\infty}{1\over t^{n-1}}T_n(t)A^n(\mu-A)^{-2n}x
=-{A^{n-1}(\mu-A)^{-n}x\over (n-1)!}, \quad x\in X.
$$
\end{itemize}
\end{theorem}

\begin{proof}
 Take $\mu>\delta>0$. For $x\in X$ define
$$
\Vert x\Vert_Y:=\sup_{t\ge0}\Vert e^{-\delta t}
(T_n(t)A^n(\mu-A)^{-n}x
+\sum_{j=0}^{n-1}{t^j\over j!}A^j(\mu-A)^{-n}x)\Vert.
$$
Note that $A(\mu-A)^{-1}=-I+\mu(\mu-A)^{-1}$ is a bounded operator on $X$ and $T_n(0)=0$, so $\Vert \cdot \Vert_Y$ is a norm on $X$ and there exists a constant $M_\delta>0$ such that
\begin{equation}\label{boundbelow}
\Vert (\mu-A)^{-n}x\Vert
\le\Vert x \Vert_Y\le M_\delta\Vert x \Vert, \quad x\in X.
\end{equation}

Let $Y$ be the Banach space obtained as the completion of $X$ in the norm
$\Vert \cdot \Vert_Y$.
By the Extrapolation Theorem \cite[Theorem 0.2]{[ANS]}, there exists a closed operator $B$ on $Y$ which generates a  $C_0$-semigroup $(S(t))_{t\ge0}\subset\mathcal B(Y)$ of positive exponential type such that $D(B^n)\hookrightarrow X\hookrightarrow Y$, $A=B_X$ where the operator $B_X$ is given by the conditions $D(B_X):=\{x\in D(B)\cap X:Bx\in X\}$,
$B_X(x):=B(x)$ ($x\in X$). Moreover, $\sigma_P(B^*)\subseteq\sigma_P(A^*)$, and also $\rho(A)=\rho(B)$ with
\begin{equation}\label{ABresolv}
(\lambda-A)^{-1}x=(\lambda-B)^{-1}x, \quad \lambda\in\rho(A)=\rho(B), x\in X;
\end{equation}
see \cite[Remark 3.1]{[ANS]}.

Let $(S_n(t))_{t\ge 0}$ be the $n$-times integrated semigroup generated by $B$ on $Y$, given by
$$
S_n(t)y:={1\over(n-1)!}\int_0^t(t-s)^{n-1}S(s)y\ ds, \quad y\in Y.
$$
Then $S_n(t)x=T_n(t)x$ for all $x\in X$ and $t\ge 0$. To see this, note that $(T_n(t))_{t\ge 0}$ and $(S_n(t))_{t\ge 0}$ are of exponential type so one can rewrite (\ref{ABresolv}) above in terms of the  Laplace transforms of $(T_n(t))_{t\ge 0}$ and $(S_n(t))_{t\ge 0}$ respectively, for $\Re\lambda$ large enough. Then it suffices to apply the uniqueness of the Laplace transform.

From the above identification between $T_n(t)$ and $S_n(t)$, it readily follows that
$$
\Vert S_n(u)x\Vert_Y\le\Vert T_n(u)\Vert\,\,\Vert x\Vert_Y
\le\omega(u)\Vert x\Vert_Y, \quad u\ge0, x\in X,
$$
which is to say, by density, that $\Vert S_n(u)\Vert\le\omega(u)$, for all $u\ge0$.

Now, by reiteration of the well known equality 
$$
\displaystyle S(t)y-y=\int_0^t BS(s)y\ ds \quad (t\ge0,\ y\in D(B)),
$$ 
we have
$$
S(t)y=S_n(t)B^ny+\sum_{j=0}^{n-1}{t^j\over j!}B^jy, \quad y\in D(B^n).
$$

Hence, for every $y\in Y$,
\begin{equation}\label{displaymu}
S(t)(\mu-B)^{-n}y=S_n(t)\left({B\over \mu-B}\right)^{n}y
+\sum_{j=0}^{n-1}{t^j\over j!}\left({B\over \mu-B}\right)^{j}(\mu-B)^{-(n-j)}y
\end{equation}
and therefore there exists a  constant $C_\mu>0$ such that
$$
\Vert S(t)(\mu-B)^{-n}\Vert_{Y\to Y}\le C_\mu\omega(t), \quad t\ge0.
$$

Then, by applying Theorem \ref{qnstable} with $U(t)=S(t)$, $B=L$ and
$R=(\mu-A)^{-n}$, we obtain
$$
\lim_{t\to\infty}{1\over\omega(t)}\Vert S(t)(\mu-B)^{-n}y\Vert_Y=0, \quad y\in Y,
$$
whence, by (\ref{boundbelow}), (\ref{ABresolv}) and (\ref{displaymu}),
$$
\begin{aligned}
0=&\lim_{t\to\infty}{1\over\omega(t)}\Vert T_n(t)A^n(\mu-A)^{-n}x
+\sum_{j=0}^{n-1}{t^j\over j!}A^j(\mu-A)^{-n}x\Vert_Y \\
&\ge
\limsup_{t\to\infty}{1\over\omega(t)}\Vert T_n(t)A^n(\mu-A)^{-2n}x
+\sum_{j=0}^{n-1}{t^j\over j!}A^j(\mu-A)^{-2n}x\Vert_X,
\end{aligned}
$$
for every $x\in X$.

Thus we get
$$
\lim_{t\to\infty}{1\over\omega(t)}T_n(t)A^n(\mu-A)^{-2n}x
=-\lim_{t\to\infty}{1\over\omega(t)}
\sum_{j=0}^{n-1}{t^j\over j!}A^j(\mu-A)^{-2n}x
$$
in $X$, and the statement follows readily.
\end{proof}

\noindent
\textit{Proof of Theorem \ref{qnstable}.}
In the setting of Theorem \ref{preResult},  let assume in addition that $0\in\rho(A)$. Since the resolvent function of $A$ is holomorphic -so continuous- in the open subset $\rho(A)\subseteq\C$ we have that
$$
\lim_{\mu\to0^+}A^n(\mu-A)^{-n}=\lim_{\mu\to0^+}(-I+\mu(\mu-A)^{-1})^n=(-1)^nI.
$$
Now, to prove (i) and (ii) of the theorem it suffices to notice that
$\sup_{t>0}\omega(t)^{-1}\Vert T_n(t)\Vert<\infty$ in both cases.
\qed

\section{Final comments and remarks}

It looks desirable to find out the behavior of a $n$-times integrated semigroup at infinity when its generator $A$ is not 
assumed to be invertible. According to the remark prior to Theorem \ref{integrastable} the existence of $\lim_{t\to\infty}T_n(x)$ 
(for $n=1$) entails invertibility of $A$.
Thus the type of convergence at infinity of $T_n(t)$, if there is some, that one can expect if $A$ is not invertible must be 
weaker than the existence of limit.

In \cite{[GMM]}, under the assumptions
$$
\sup_{t>0} t^{-n}\Vert T_n(t)\Vert<\infty \hbox{ and }
\lim_{t\to0^+}n! t^{-n} T_n(t)x=x \quad(x\in X),
$$
it has been proved that
$$
\lim_{t\to\infty} t^{-n} T_n(t)\pi_n(f)=0, \quad f\in \mathcal {\frak S}_n,
$$
in the operator norm, where $\mathcal {\frak S}_n$ is the subspace of functions of
$\mathcal T_+^{(n)}(t^n)$ which are of spectral synthesis in
$\mathcal T^{(n)}(\vert t\vert^n)$
with respect to the subset $i\sigma(A)\cap\R$, and $\pi_n(f)=(-1)^n\int_0^\infty f^{(n)}(t) T_n(t)\ dt$.
Here, $\mathcal T^{(n)}(\vert t\vert^n)$ is the convolution Banach algebra obtained as the completion 
of the Schwarz class in the norm $f\mapsto\int_{-\infty}^\infty\vert f^{(n)}(t)\vert\, \vert t\vert^n\ dt$, and $\mathcal T_+^{(n)}(t^n)$ 
is the restriction of $\mathcal T^{(n)}(\vert t\vert^n)$ on $(0,\infty)$.
This result is an extension of the Esterle-Strouse-V$\tilde {\rm u}$-Zouakia theorem, which corresponds to the case $n=0$; 
see \cite{[ESZ]} and \cite{[V1]}. In \cite{[ESZ]} it is shown that, under the assumptions that $\sigma(A)\cap i\R$ is countable and
$\sigma_P(A^*)\cap i\R=\emptyset$, the subspace
$\pi_0(\mathcal {\frak S}_0)X$ is dense in $X$ so one gets another way --for it is different from the original one-- to establish 
the Arendt-Batty-Lyubich-V$\tilde {\rm u}$ theorem. The proof of that density is attained by methods of harmonic analysis.

We wonder if in the case when $A$ is not invertible the argument considered in \cite{[ESZ]} to deduce the 
Arendt-Batty-Lyubich-V$\tilde {\rm u}$ theorem works for $n$-times integrated semigroups; that is, 
if $\pi_n(\mathcal {\frak S}_n)X$ is dense in $X$ (under the conditions $\sigma(A)\cap i\R$ countable and 
$\sigma_P(A^*)\cap i\R=\emptyset$). This would give us the ergodic type property
\begin{equation}\label{homogestable}
\lim_{t\to\infty} t^{-n} T_n(t)x=0, \quad x\in X.
\end{equation}
Notice that (\ref{homogestable}) is a consequence of Theorem \ref{qnstable} (i) when $A$ is invertible;
on the other hand, the ergodicity of a $n$-times integrated semigroup $(T_n(t))_{t\ge 0}$ 
such that $\sup_{t\ge1}t^{-n}\Vert T_n(t)\Vert<\infty$ is characterized in \cite{[EM]} in terms of Abel-ergodicity or/and ergodic decompositions of the Banach space $X$.
Such an approach will be considered in a forthcoming paper.



\end{document}